\documentclass[a4paper]{article}
\usepackage{amssymb}
\usepackage{graphicx}
\usepackage{amsmath}
\usepackage{amsfonts}
\usepackage{amsthm}
\usepackage{mathrsfs}
\usepackage{dsfont}
\usepackage{indentfirst}
\usepackage{esint}
\usepackage{hyperref}
\usepackage{yhmath}

\numberwithin{equation}{section}

\newcommand{\Sec}{\mathrm{Sec}}

\newcommand{\diam}{\mathrm{diam}}

\newcommand{\Ric}{\mathrm{Ric}}
\newcommand{\Span}{\mathrm{Span}}

\title{\Large \bf \boldmath\ \\ Geometric analysis aspects of infinite semiplanar graphs with nonnegative curvature II} 

\author{\large  Bobo Hua$^\ast$\ \ \ \  J\"{u}rgen Jost$^{\dag}$ \ \ \ \ } 

\date{}

\par

\begin{document}

\maketitle

\renewcommand{\thefootnote}{\fnsymbol{footnote}}

\footnotetext{\hspace*{-5mm} \begin{tabular}{@{}r@{}p{13.4cm}@{}}
$^\ast$ $^{\dag}$ & Max Planck Institute for Mathematics in the Sciences, Leipzig, 04103, Germany.\\
$^\ast$ $^{\dag}$ & The research leading to these results has
received funding from the European Research Council under the
European Union's Seventh Framework Programme (FP7/2007-2013) / ERC
grant agreement
n$^\circ$~267087.\\
&{Email: bobohua@mis.mpg.de, jost@mis.mpg.de}\\
&The Mathematics Subject Classification 2010: 31C05, 05C81, 05C63.
\end{tabular}}

\renewcommand{\thefootnote}{\arabic{footnote}}

\newtheorem{theorem}{Theorem}[section]
\newtheorem{con}[theorem]{Conjecture}
\newtheorem{lemma}[theorem]{Lemma}
\newtheorem{corollary}[theorem]{Corollary}
\newtheorem{definition}[theorem]{Definition}
\newtheorem{conv}[theorem]{Assumption}
\newtheorem{remark}[theorem]{Remark}

\bigskip

\begin{abstract}  In a previous paper \cite{HJL}, we have applied
  Alexandrov geometry methods to study infinite semiplanar graphs with
  nonnegative combinatorial curvature. We \cite{HJL} proved the weak relative
  volume comparison and the Poincar\'e inequality on these graphs to
  obtain a dimension estimate for polynomial growth harmonic functions
  which is asymptotically quadratic in the growth rate. In the present
  paper, instead of using volume comparison on the graph, we translate
  the problem to a polygonal surface by filling polygons into the
  graph with edge lengths 1. This polygonal surface then is an
  Alexandrov space of nonnegative curvature. On this  Alexandrov
  space, we then obtain the optimal  dimension estimate
  for polynomial growth harmonic functions on the graphs. The estimate
  is
  linear in the growth rate. From a harmonic function on the graph, we
  construct a function on the corresponding Alexandrov surface that is
  not necessarily harmonic, but satisfies crucial estimates. Finally,
  a combinatorial argument  controls the maximal facial degree of
  the polygonal surface.
\medskip
\end{abstract}
\section{Introduction}
This paper is the second one in a series  studying geometric analysis aspects of infinite graphs with nonnegative curvature. We refine the argument in Hua-Jost-Liu \cite{HJL} and introduce a new observation to obtain the asymptotically optimal dimension estimate of the space of polynomial growth harmonic functions on such graphs.

In 1975, Yau \cite{Y1} proved the Liouville theorem for harmonic
functions on Riemannian manifolds with nonnegative Ricci
curvature. Soon after Cheng-Yau \cite{CheY} obtained the gradient
estimate for positive harmonic functions which implies that sublinear
growth harmonic functions on these manifolds are constant. Then Yau
\cite{Y2,Y3} conjectured that the space of polynomial growth harmonic
functions with growth rate less than or equal to $d$ on Riemannian
manifolds with nonnegative Ricci curvature is of finite
dimension. Li-Tam \cite{LT} and Donnelly-Fefferman \cite{DF}
independently solved the conjecture for 2-dimensional manifolds. Then
Colding-Minicozzi \cite{CM1,CM2,CM3} gave the affirmative answer for
any dimension by using the volume comparison property and the Poincar\'e
inequality. Later, Li \cite{L1} and Colding-Minicozzi \cite{CM4} simplified the proof by the mean value inequality. The dimension estimates in \cite{CM3,CM4,L1} are
asymptotically optimal. In the wake of this result, many
generalizations on manifolds \cite{W,T,STW,LW3,LW1,LW2,CW,KLe,Le} and on
singular spaces \cite{D1,Kl,Hu1,Hu2} followed. In this paper, we obtain the optimal dimension estimate which is linear in $d$ rather than quadratic in $d$ as in \cite{HJL}.

Let us now describe the results in more details. The combinatorial curvature
for planar graphs was introduced by \cite{St,G,I} and studied by many authors
\cite{H,Woe,DM,CC,Ch,SY,RBK,BP1,BP2,K1,K2,K3}. Let $G=(V,E,F)$ be a (called
semiplanar) graph embedded in a 2-manifold such that each face is
homeomorphic to a closed disk with finite edges as the boundary. Let
$S(G)$ be the regular polygonal surface obtained by assigning
length one to every edge and filling regular polygons in the faces of $G.$ The combinatorial curvature at the vertex $x$ is defined as
$$\Phi(x)=1-\frac{d_x}{2}+\sum_{\sigma\ni x}\frac{1}{\deg(\sigma)},$$
where $d_x$ is the degree of the vertex $x$, $\deg(\sigma)$ is the
degree of the face $\sigma$, and the sum is taken over all faces
incident to $x$ (i.e. ${x\in \sigma}$). The idea of this definition
is to measure the
difference of $2\pi$ and the total angle $\Sigma_x$ at the vertex $x$
on the regular polygonal surface $S(G)$ equipped with a metric structure obtained from replacing each face of $G$ with a regular polygon of side length one and gluing them along the common edges. That is,
$$2\pi \Phi(x)=2\pi-\Sigma_x.$$
It was proved in \cite{HJL} that $G$ has nonnegative combinatorial
curvature everywhere if and only if the corresponding regular polygonal surface
$S(G)$ is an Alexandrov space with nonnegative sectional curvature,
i.e. $\Sec S(G)\geq 0$ (or $\Sec G\geq0$ for short). This class of graphs includes all regular tilings of the plane (see \cite{GS}) and more general graphs (see \cite{Ch,HJL}).

For the basic facts of Alexandrov spaces, readers are referred to \cite{BGP,BBI}. In this paper, we only consider 2-dimensional Alexandrov spaces with nonnegative curvature (in particular convex surfaces). Let $G$ be a semiplanar graph with $\Sec G\geq0,$ then $X:=S(G)$ is a 2-dimensional Alexandrov space with nonnegative curvature. We denote by $d$ the intrinsic metric on $X,$ by $B_R(p):=\{x\in X\mid d(x,p)\leq R\}$ the closed geodesic ball on $X,$ and by $|B_R(p)|:=H^2(B_R(p))$ the volume of $B_R(p),$ i.e. 2-dimensional Hausdorff measure of $B_R(p),$ for some $p\in X, R>0.$ The well known Bishop-Gromov volume comparison holds on $X$ (see \cite{BBI}) that for any $p\in X, 0<r<R$, we have
\begin{equation}\label{RVC1}\frac{|B_R(p)|}{|B_r(p)|}\leq\left(\frac{R}{r}\right)^2,\end{equation}
 \begin{equation}\label{VD1}|B_{2R}(p)|\leq4 |B_R(p)|,\end{equation}
We call (\ref{RVC1}) the relative volume comparison and (\ref{VD1}) the volume doubling property. The Poincar\'e inequality was proved in \cite{KMS, Hu1} on Alexandrov spaces. For any $p\in X, R>0$ and any Lipschitz function $u$ on $X,$ \begin{equation}\label{PI1}\int_{B_R(p)}|u-u_{B_R}|^2\leq CR^2\int_{B_R(p)}|\triangledown u|^2,\end{equation} where $u_{B_R}=\frac{1}{|B_R(p)|}\int_{B_R(p)}u,$ and $|\triangledown u|$ is the a.e. defined norm of the gradient of $u.$

 It has been shown in \cite{HJL} that $G$ inherits some geometric  estimates  from those of $X:=S(G).$ For any $p\in G$ and $R>0,$ we denote by $d^G$ the distance on the graph $G$, by $B_R^G(p)=\{x\in G: d^{G}(p,x)\leq R\}$ the closed geodesic ball on $G$ and by $|B_R^G(p)|:=\sum_{x\in B_R^G(p)}d_x$ the volume of $B_R^G(p).$ Let $D:=D_G$ denote the maximal degree of the faces in $G$, i.e. $D:=D_G:=\sup_{\sigma\in F}\deg(\sigma)$ which is finite by \cite{CC}. Then the weak relative volume comparison (\ref{RVCG1}) and the volume doubling property (\ref{VDG1}) were obtained in \cite{HJL} for $\Sec G\geq 0$. For any $p\in G, 0<r<R,$
 \begin{equation}\label{RVCG1}\frac{|B_R^G(p)|}{|B_r^G(p)|}\leq C(D)\left(\frac{R}{r}\right)^2,\end{equation}
 \begin{equation}\label{VDG1}|B_{2R}^G(p)|\leq C(D) |B_R^G(p)|,\end{equation}
where $C(D)$ are constants only depending on $D.$ The Poincar\'e inequality on $G$ was also obtained in \cite{HJL}.
There exist two constants $C(D)$ and $C$ such that for any $p\in G, R>0, f:B_{CR}^G(p)\rightarrow \mathds{R}$, we have
\begin{equation}\label{PIG1}\sum_{x\in B_R^G(p)}(f(x)-f_{B_R})^2d_x\leq C(D)R^2\sum_{x,y\in B_{CR}^G(p);x\sim y}(f(x)-f(y))^2,\end{equation}
where $f_{B_R}=\frac{1}{|B_R^G(p)|}\sum_{x\in B_R^G(p)}f(x)d_x$, and $x\sim y$ means that $x$ and $y$ are neighbors in $G$.

A function $f$ on $G$ is called discrete harmonic (see \cite{G3,DK,Chu}) if for $\forall x\in G,$ $$Lf(x):=\frac{1}{d_x}\sum_{y\sim x}(f(y)-f(x))=0.$$
Let $G$ be a semiplanar graph with nonnegative curvature and
$H^d(G):=\{u:G\rightarrow\mathds{R}\mid Lu\equiv0, |u(x)|\leq
C(d^G(p,x)+1)^d\}$ which is the space of polynomial growth harmonic
functions with growth rate less than or equal to $d$ on $G.$ For a Riemannian manifold $M$, let $H^d(M):=\{u:M\rightarrow\mathds{R}\mid\Delta_Mu=0,|u(x)|\leq C(d(p,x)+1)^d\}.$ In Riemannian case, Colding-Minicozzi \cite{CM2} used the volume doubling property (\ref{VD1}) and the
Poincar\'e inequality (\ref{PI1}) to conclude the finite dimensionality of the space of polynomial growth harmonic functions $H^d(M)$ and get a rough dimension estimate. Then Colding-Minicozzi \cite{CM3} used the relative volume comparison and the Poincar\'e inequality to obtain the asymptotically optimal dimension estimate, i.e. $\dim H^d(M)\leq C(n)d^{n-1},$ for $d\geq1,$ $\Ric M^n\geq0$. Li \cite{L1} and Colding-Minicozzi \cite{CM4} obtained the optimal dimension estimate by the mean value inequality. In the graph case, the volume doubling property (\ref{VDG1}) and the Poincar\'e inequality (\ref{PIG1}) imply that $\dim H^d(G)\leq C(D)d^{v(D)}$ where $C(D)$ and $v(D)$ depending on the maximal facial degree $D$ (see \cite{D1}). Hua-Jost-Liu \cite{HJL} used the weak relative volume comparison (\ref{RVCG1}) to obtain the estimate $\dim H^d(G)\leq C(D)d^2.$ It is obviously not optimal. But it is hard to obtain the optimal dimension estimate on the graph $G$ since the constant!
  $C(D)$ in the weak relative volume comparison (\ref{RVCG1}) may not be close to $1$.

Here comes the key observation. Since the relative volume comparison
(\ref{RVC1}) on $X$ is as nice as in the case of Riemannian manifolds,
we do the dimension estimate argument for $H^d(G)$ on $X.$ For any
(discrete) harmonic function $f$ on $G,$ we extend it to a function
$\bar{f}$ defined on $X$ with controlled behavior (see
(\ref{DFF1})(\ref{DFF2})). But in general, the extended function
$\bar{f}$ may not be harmonic on $X$ anymore, nor will  $\bar{f}^2$ be subharmonic. However, since the original harmonic function $f$ satisfies the mean value inequality on $G$ (see Lemma \ref{MVIOG}), the extended function $\bar{f}$ satisfies the mean value inequality in the large.
\begin{theorem}[Mean value inequality on $X$]\label{MVIIL} Let $G$ be a semiplanar graph with $\Sec G\geq0.$ Then there exist constants $R_1(D),$ $C_2(D)$ such that for any $p\in X, R\geq R_1(D)$ and any harmonic function $f$ on $G$ we have
\begin{equation}\label{MVIOX}\bar{f}^2(p)\leq\frac{C_2}{|B_R(p)|}\int_{B_R(p)}\bar{f}^2.\end{equation}
\end{theorem}

Let $P^d(X):=\{u:X\rightarrow \mathds{R}\mid |u(x)|\leq C(d(p,x)+1)^d\}$ denote the space of polynomial growth functions on $X$ with growth rate less than or equal to $d.$ Since the extending map
\begin{equation}\label{ETM1}E:H^d(G)\rightarrow P^d(X),\ f\mapsto
  Ef=\bar{f},\end{equation} is an injective linear operator, it suffices to get the dimension estimate for the image $E(H^d(G)).$ Combining the relative volume comparison (\ref{RVC1}) and the mean value inequality (\ref{MVIOX}), we obtain the optimal dimension estimate for $E(H^d(G)).$

Although we have to pay for extending map $E$ by the loss of harmonicity, it preserves the mean value property which is sufficient for our application. We adopt the argument of the mean value inequality (see \cite{L1,L2,CM4}) to get optimal dimension estimate. In addition, by the special structure of the graph with $\Sec G\geq0$ and $D\geq 43,$ we \cite{HJL} obtained that for any $d>0$,
$$\dim H^d(G)=1,$$ which implies the final theorem of the paper.

\begin{theorem}\label{MTAL} Let $G$ be a semiplanar graph with $\Sec G\geq 0$. Then for any $d\geq 1$,
$$\dim H^d(G)\leq Cd,$$ where $C$ is an absolute constant.
\end{theorem}

From a superficial glance, it might look as if
 polynomial growth harmonic functions on Riemannian manifolds
 (continuous objects) and those on graphs (discrete ones) are very
 similar and might succumb to an analogous treatment. While our work
 is indeed inspired by certain analogies, there are also some
 important differences which necessitate new ideas which we now wish
 to summarize. Firstly, the unique continuation property for
 (discrete) harmonic functions on graphs fails, leaving us with the
 problem of verifying the inner product property of the bilinear form
 $L^2(B_R)$ on $H^d(G)$ where $B_R$ is the geodesic ball of radius $R$
 in a graph $G$ (see \eqref{innerp1}). We use a lemma in \cite{Hu2}
 (see Lemma \ref{innerprod} in this paper) to overcome this
 difficulty. Secondly, the constant $C(D)$ in the relative volume
 comparison \eqref{RVCG1} on semiplanar graphs with nonnegative
 curvature is not necessarily close to 1. Even on manifolds, it is
 still an open problem to obtain the optimal dimension estimate by
 using \eqref{RVCG1} and \eqref{PIG1}. In this paper, we find an
 argument which transforms the discrete harmonic functions on the
 semiplanar graph $G$ with nonnegative curvature to functions on the
 polygonal surface $S(G)$ that satisfy the mean value inequality. This
 crucial step enables us to transfer the argument to $S(G)$ where we
 have a nice volume comparison \eqref{RVC1} and to obtain the optimal dimension estimate of $H^d(G)$. Thirdly, the combinatorial obstruction for semiplanar graphs with a large face (i.e. $D\geq 43$)  makes the dimension estimate independent of the parameter $D.$

\section{Preliminaries and Notations}
We recall the definition of semiplanar graphs in \cite{HJL}.
\begin{definition} A graph $G=(V,E)$ is called semiplanar if it can be
  embedded into a connected 2-manifold $S$ without self-intersections
  of edges and such that each face is homeomorphic to the closed disk with finite edges as the boundary.
\end{definition}

Let $G=(V,E,F)$ denote the semiplanar graph with the set of
vertices, $V$, edges, $E$ and faces, $F$. Edges and faces are regarded
as closed subsets of $S$, and two objects from $V, E, F$ are called
incident if one is a proper subset of the other. We
always assume that the surface $S$ has no boundary and the graph $G$ is a simple graph, i.e. without loops and
multi-edges. We denote by $d_x$ the degree of the vertex $x\in G$ and
by $\deg(\sigma)$  the degree of the face $\sigma\in F$, i.e. the
number of edges incident to $\sigma$. Further, we assume that $3\leq
d_x< \infty$ and $3\leq \deg(\sigma)<\infty$ for each vertex $x$ and
face $\sigma$, which means that $G$ is a locally finite graph. For
each semiplanar graph $G=(V,E,F)$, there is a unique metric space,
denoted by $S(G)$, which is obtained from replacing each face of $G$
by a regular polygon of side length one with the same facial degree and gluing the faces along the common edges in $S$. $S(G)$ is called the regular polygonal surface of the semiplanar graph $G$.

For a semiplanar graph $G$, the combinatorial curvature at each vertex $x\in G$ is defined as
$$\Phi(x)=1-\frac{d_x}{2}+\sum_{\sigma\ni x}\frac{1}{\deg(\sigma)},$$ where the sum is taken over all the faces incident to $x$. In this paper, we only consider semiplanar graphs with nonnegative combinatorial curvature. It was proved in \cite{HJL} that a semiplanar graph $G$ has nonnegative combinatorial curvature everywhere if and only if the regular polygonal surface $S(G)$ is an Alexandrov space with nonnegative curvature, denoted by $\Sec G\geq0$ or $\Sec S(G)\geq0$.

For Alexandrov spaces and Alexandrov geometry, readers are referred to \cite{BGP,BBI}.
A curve $\gamma$ in a metric space $(X,d)$ is a continuous map $\gamma : [a,b]\rightarrow X.$ The length of a curve $\gamma$ is defined as
$$L(\gamma)=\sup\left\{\sum_{i=1}^Nd(\gamma(y_{i-1}),\gamma(y_i)):
\text{any\ partition }a=y_0<y_1<\ldots<y_N=b\right\}.$$ A curve $\gamma$ is called rectifiable if $L(\gamma)<\infty.$ Given $x, y\in X$, denote by $\Gamma(x,y)$ the set of rectifiable curves joining $x$ and $y$. A metric space $(X,d)$ is called a length space if $d(x,y)=\inf_{\gamma\in\Gamma(x,y)}\{L(\gamma)\},$ for any $x,y \in X$, where $d$ is called the intrinsic metric on $X.$ A curve $\gamma:[a,b]\rightarrow X$ is called a geodesic if $d(\gamma(a),\gamma(b))=L(\gamma).$ It is always true by the definition of the length of a curve that $d(\gamma(a),\gamma(b))\leq L(\gamma).$ A geodesic is a shortest curve (or shortest path) joining the two end points. A geodesic space is a length space $(X,d)$ satisfying that for any $x,y\in X,$ there is a geodesic joining $x$ and $y$.

Denote by $\Pi_{\kappa},$ $\kappa\in \mathds{R}$ the model space which is a 2-dimensional, simply connected space form of constant curvature $\kappa$. Typical ones are $$\Pi_{\kappa}=\left\{
\begin{array}{ll}
\mathds{R}^2,&\kappa=0\\
S^2,&\kappa=1\\
\mathds{H}^2,&\kappa=-1
\end{array}\right..$$ In a geodesic space $(X,d)$, we denote by $\gamma_{xy}$ one of the geodesics joining $x$ and $y$, for $x,y\in X$. Given three points $x, y , z \in X,$ denote by $\triangle_{xyz}$ the geodesic triangle with edges $\gamma_{xy}, \gamma_{yz}, \gamma_{zx}.$ There exists a unique (up to an isometry) geodesic triangle, $\triangle_{\bar x\bar y\bar z}$, in $\Pi_{\kappa}$ ($d(x,y)+d(y,z)+d(z,x)<\frac{2\pi}{\sqrt{\kappa}}$ if ${\kappa>0}$) such that $d(\bar x, \bar y)=d(x,y), d(\bar y, \bar z)=d(y,z)$ and $d(\bar z, \bar x)=d(z,x).$ We call $\triangle_{\bar x\bar y\bar z}$ the comparison triangle in $\Pi_{\kappa}$.

\begin{definition} A complete geodesic space $(X,d)$ is called an Alexandrov space with sectional curvature bounded below by $\kappa$ (Sec$X\geq \kappa$ for short) if for any $p\in X$, there exists a neighborhood $U_p$ of $p$ such that for any $x,y,z\in U_p$, any geodesic triangle $\triangle_{xyz}$, and any $w\in\gamma_{yz}$, letting $\bar w\in \gamma_{\bar y \bar z}$ be in the comparison triangle $\triangle_{\bar x\bar y\bar z}$ in $\Pi_{\kappa}$ satisfying $d(\bar y,\bar w)=d(y,w)$ and $d(\bar w,\bar z)=d(w,z)$,  we have $$d(x,w)\geq d(\bar x,\bar w).$$
\end{definition}

In other words, an Alexandrov space $(X,d)$ is a geodesic space which
locally satisfies the Toponogov triangle comparison theorem for the sectional curvature.
It was proved in \cite{BGP} that the Hausdorff dimension of an Alexandrov space $(X,d)$, $\dim_H(X)$, is an integer or infinity. In this paper, we only consider 2-dimensional Alexandrov spaces with $\Sec X\geq0$.

Let $G$ be a semiplanar graph with nonnegative combinatorial curvature. Let $X:=S(G)$ be the regular polygonal surface of $G$ with the intrinsic metric $d$. Then $\Sec X\geq0.$ Let $B_R(p)$ denote the closed geodesic ball centered at $p\in X$ of radius $R>0$, i.e. $B_R(p)=\{x\in X: d(p,x)\leq R\},$ $|B_R(p)|:=H^2(B_R(p))$ denote the volume of $B_R(p),$ i.e. $2$-dimensional Hausdorff measure of $B_R(p).$ The well known Bishop-Gromov volume comparison theorem holds on Alexandrov spaces \cite{BBI}.

\begin{lemma} Let $(X,d)$ be an $2$-dimensional Alexandrov space with nonnegative curvature, i.e. $\Sec X\geq0.$ Then for any $p\in X, 0<r<R,$ it holds that
\begin{equation}\label{RVCA1}\frac{|B_R(p)|}{|B_r(p)|}\leq\left(\frac{R}{r}\right)^2.\end{equation}
\begin{equation}\label{VDA1}|B_{2R}(p)|\leq4 |B_R(p)|.\end{equation}
\end{lemma}
We call (\ref{RVCA1}) the relative volume comparison and (\ref{VDA1}) the volume doubling property.

For any precompact domain $\Omega\subset X,$ we denote by $Lip(\Omega)$ the set of Lipschitz functions on $\Omega$. For any $f\in Lip(\Omega)$, the $W^{1,2}$ norm of $f$ is  defined as
$$\|f\|_{W^{1,2}(\Omega)}^2=\int_{\Omega}f^2+\int_{\Omega}|\nabla f|^2.$$ The $W^{1,2}$ space on $\Omega$, denoted by $W^{1,2}(\Omega)$, is the completion of $Lip(\Omega)$ with respect to the $W^{1,2}$ norm. A function $f\in W^{1,2}_{loc}(X)$ if for any precompact domain $\Omega\subset\subset X$, $f|_{\Omega}\in W^{1,2}(\Omega)$. The Poincar\'e inequality was proved in \cite{KMS,Hu1}.
\begin{lemma} Let $(X,d)$ be an $2$-dimensional Alexandrov space
  with $\Sec X\geq 0$ and $u\in W^{1,2}_{loc}(X),$ then
\begin{equation}\label{PIA1}\int_{B_R(p)}|u-u_{B_R}|^2\leq C(n)R^2\int_{B_R(p)}|\triangledown u|^2,\end{equation} where $u_{B_R}=\frac{1}{|B_R(p)|}\int_{B_R(p)}u.$
\end{lemma}

Let $G=(V,E,F)$ be a semiplanar graph with nonnegative combinatorial curvature and $X:=S(G)$ be the regular polygonal surface of $G$. Then it is straightforward that $3\leq d_x\leq 6$ for $\forall x\in G,$ i.e. $G$ has bounded degree. We denote by $D:=D_G:=\sup\{\deg(\sigma):\sigma\in F\}$ the maximal degree of faces in $G,$ which is a very important parameter in our discussion (it is finite by Gauss-Bonnet formula in \cite{DM,CC}).  For any $x,y\in G,$ they are called neighbors, denoted by $x\sim y,$ if there is an edge in $E$ connecting $x$ and $y.$ There is a natural metric on the graph $G,$ $d^G(x,y):=\inf\{k: \exists x=x_0\sim\cdots\sim x_k=y\},$ i.e. the length of the shortest path connecting $x$ and $y$ by assigning each edge the length one. Lemma 3.1 in \cite{HJL} implies that the two metrics, $d^G$ and $d,$ on $G$ are bi-Lipschitz equivalent, i.e. there exists a universal constant $C$ such that for any $x,y\in G$
\begin{equation}\label{LEM2}Cd^{G}(x,y)\leq d(x,y)\leq d^{G}(x,y).\end{equation}

For any $p\in G$ and $R>0,$ we denote by $B_R^G(p)=\{x\in G: d^{G}(p,x)\leq R\}$ the closed geodesic ball in the graph $G,$ by $|B_R^G(p)|:=\sum_{x\in B_R(p)}d_x$ the volume of $B_R^G(p),$ and by $\sharp B_R^G(p)$ the number of vertices in the closed geodesic ball $B_R^G(p).$ Since $3\leq d_x\leq 6$ for any $x\in G,$ $|B_R^G(p)|$ and $\sharp B_R^G(p)$ are equivalent up to a constant, i.e. $3 \sharp B_R^G(p)\leq|B_R^G(p)|\leq 6 \sharp B_R^G(p),$ for any $p\in G$ and $R>0.$ The following volume comparison on $G$ was proved in \cite{HJL} by the relative volume comparison (\ref{RVCA1}) on $X.$

\begin{lemma}\label{GRV} Let $G=(V,E,F)$ be a semiplanar graph with $\Sec G\geq0$. Then there exists a constant $C(D)$ depending on $D,$ such
that for any $p\in G$ and $0<r<R,$ we have
\begin{equation}\label{GRV1}\frac{|B_R^G(p)|}{|B_r^G(p)|}\leq C(D)\left(\frac{R}{r}\right)^2.\end{equation}
\begin{equation}\label{GVD1}|B_{2R}^G(p)|\leq C(D) |B_R^G(p)|.\end{equation}
\end{lemma}

We call (\ref{GRV1}) the weak relative volume comparison and (\ref{GVD1}) the volume doubling property on $G.$ The Poincar\'e inequality on $G$ was also obtained in \cite{HJL} by the Poincar\'e inequality (\ref{PIA1}).

\begin{lemma} Let $G$ be a semiplanar graph with $\Sec G\geq 0$. Then there exist two constants $C(D)$ and $C$ such that for any $p\in G, R>0, f:B_{CR}^G(p)\rightarrow \mathds{R}$, we have
\begin{equation}\label{PII1}\sum_{x\in B_R^G(p)}(f(x)-f_{B_R^G})^2d_x\leq C(D)R^2\sum_{x,y\in B_{CR}^G(p);x\sim y}(f(x)-f(y))^2,\end{equation}
where $f_{B_R^G}=\frac{1}{|B_R^G(p)|}\sum_{x\in B_R^G(p)}f(x)d_x.$
\end{lemma}

\section{Mean Value Inequality}
In this section, we extend each harmonic function on the semiplanar
graph $G$ with nonnegative combinatorial curvature to a function on
$X:=S(G)$ which is  almost harmonic in the sense that it satisfies the mean value inequality on $X.$

For any $\Omega\subset G$ and $x\in G,$ we define $d^G(x,\Omega):=\inf\{d^G(x,y)\mid y\in \Omega\}.$ We denote $\partial\Omega:=\{x\in G\mid d(x,\Omega)=1\}$ and $\bar{\Omega}:=\Omega\cup\partial \Omega.$ The function $f$ is called harmonic on $\Omega$ if $f:\bar{\Omega}\rightarrow \mathds{R}$ satisfies
$$Lf(x):=\frac{1}{d_x}\sum_{y\sim x}(f(y)-f(x))=0,$$ for any $x\in\Omega,$ where $L$ is called the Laplacian operator.

Since the volume doubling property (\ref{GVD1}) and the Poincar\'e inequality (\ref{PII1}) are obtained on the semiplanar graph with nonnegative combinatorial curvature, the Moser iteration can be carried out (see \cite{D2,HS}).

\begin{lemma}[Harnack inequality]\label{HNKIG} Let $G$ be a semiplanar graph with $\Sec G\geq0$. Then there exist constants $C_1(D)$ and $C_2(D)$ such that for any $p\in G,$ $R\geq1$ and any positive harmonic function $f$ on $B_{C_1R}^G(p)$ we have
\begin{equation}\label{HNKIG1}\max_{B_R^G(p)}f\leq C_2 \min_{B_R^G(p)}f.\end{equation}
\end{lemma}

The mean value inequality is one part of the Moser iteration (see also \cite{CoG}).
\begin{lemma}[Mean value inequality on graphs]\label{MVIOG}
Let $G$ be a semiplanar graph with $\Sec G\geq0$. Then there exist two constants $C_1(D)$ and $C_2(D)$ such that for any $R>0,p\in G,$ any harmonic function $f$ on $B_{C_1R}^G(p),$ we
have
\begin{equation}\label{MVI}f^2(p)\leq\frac{C_2}{|B_{C_1R}^G(p)|}\sum_{x\in B_{C_1R}^G(p)}f^2(x)d_x.\end{equation}
\end{lemma}

In the following process, we extend each function defined on $G$ to the function $\bar{f}$ defined on $X:=S(G)$ with controlled behavior. Let $f$ be a function on $G$, $f:G\rightarrow \mathds{R},$ $G_1$ be the 1-dimensional simplicial complex of $G$ by assigning each edge the length one. Step one is the linear interpolation, i.e. $f$ is extended to a piecewise linear function on $G_1,$ $f_1:G_1\rightarrow \mathds{R}.$ In step two, we extend $f_1$ to a function defined on each face of $G.$ For any regular $n$-polygon $\triangle_n$ of side length one, there is a bi-Lipschitz map $$L_n:\triangle_n\rightarrow B_{r_n},$$ where $B_{r_n}$ is the circumscribed circle of $\triangle_n$ of radius $r_n=\frac{1}{2\sin\frac{\alpha_n}{2}}$ (for $\alpha_n=\frac{2\pi}{n}$). Without loss of generality, we may assume that the origin $\underline{o}=(0,0)$ of $\mathds{R}^2$ is the barycenter of $\triangle_n,$ the point $(x,y)=(r_n,0)\in \mathds{R}^2$ is a vertex of $\triangle_n,$ and $B_{r_n}=B_!
 {r_n}(\underline{o}).$  Then in  polar coordinates, $L_n$ reads
$$L_n:\triangle_n \ni (r,\theta)\mapsto (\rho, \eta)\in B_{r_n}(\underline{o}),$$ where for $\theta\in[j\alpha_n,(j+1)\alpha_n],\ j=0,1,\cdots,n-1,$
$$\left\{
\begin{array}{ll}
\rho=\frac{r\cos\big(\theta-(2j+1)\frac{\alpha_n}{2}\big)}{\cos\frac{\alpha_n}{2}}&\\
\eta=\theta&
 \end{array}\right..$$ It maps the boundary of $\triangle_n$ to the boundary of $B_{r_n}(\underline{o})$. Direct calculation shows that $L_n$ is a
 bi-Lipschitz map, i.e. for any $x,y \in \triangle_n$ we have $C_1|x-y|\leq|L_nx-L_ny|\leq C_2|x-y|,$ where $C_1$ and $C_2$ do not depend on $n.$
Then for any $\sigma\in F,$ we denote $\sigma:=\triangle_n$ where $n:=\deg(\sigma).$ Let $g:B_{r_n}(\underline{o})\rightarrow \mathds{R}$ satisfy the following boundary value problem
\begin{equation}\label{DFF1}\left\{
\begin{array}{ll}
\Delta g=0,&in\ \mathring{B}_{r_n}(\underline{o})\\
g|_{\partial B_{r_n}(\underline{o})}=f_1\circ L_n^{-1}&
\end{array}\right.,\end{equation} where $\mathring{B}_{r_n}(\underline{o})$ is the open disk.
Then we define $\bar{f}:X\rightarrow \mathds{R}$ as
\begin{equation}\label{DFF2}\bar{f}|_{\sigma}=g\circ L_n,\end{equation} for any $\sigma\in F.$
It is easy to see that $\bar f$ is continuous function (actually it is in $W^{1,2}_{loc}(X)$).

We improve the estimates in \cite{HJL} to control the behavior of $\bar{f}$. Let $B_1$ be the closed unit disk in $\mathds{R}^2.$ For completeness, we give the proof here.

\begin{lemma}\label{B1ECL} For any Lipschitz function $h:\partial B_1\rightarrow \mathds{R},$ let $g:B_1\rightarrow\mathds{R}$ satisfy the following boundary value problem
$$\left\{
\begin{array}{ll}
\Delta g=0,&in\ \mathring{B}_1\\
g|_{\partial B_1}=h&
\end{array}\right..$$ Then we have
$$\int_{B_1}|\nabla g|^2\leq\int_{\partial B_1}h_{\theta}^2,$$
$$\int_{\partial B_1}h^2\leq C(\epsilon)\int_{B_1}g^2+\epsilon\int_{\partial B_1}h_{\theta}^2,$$ where $h_{\theta}=\frac{\partial h}{\partial \theta},$ $\epsilon$ is small.
\end{lemma}
\begin{proof}
Let $\frac{1}{\sqrt{2\pi}}, \frac{\sin n\theta}{\sqrt{\pi}}, \frac{\cos n\theta}{\sqrt{\pi}}$ (for $n=1,2,\cdots$) be the orthonormal basis of
$L^2(\partial B_1).$ Then $h:\partial B_1\rightarrow \mathds{R}$ can be represented in $L^2(\partial B_1)$ by
$$h(\theta)=a_0\frac{1}{\sqrt{2\pi}}+\sum_{i=1}^{\infty}\left( a_n\frac{\cos n\theta}{\sqrt{\pi}}+b_n\frac{\sin n\theta}{\sqrt{\pi}}\right).$$
So the harmonic function $g$ with boundary value $h$ is
$$g(r,\theta)=a_0\frac{1}{\sqrt{2\pi}}+\sum_{i=1}^{\infty}\left( a_nr^n\frac{\cos n\theta}{\sqrt{\pi}}+b_nr^n\frac{\sin n\theta}{\sqrt{\pi}}\right).$$

Since $\Delta g=0,$ we have $\Delta g^2=2|\nabla g|^2,$ then
$$\int_{B_1}|\nabla g|^2=\frac{1}{2}\int_{B_1}\Delta g^2=\frac{1}{2}\int_{\partial B_1}\frac{\partial g^2}{\partial r},$$ which follows from integration by parts.
So that $$\int_{B_1}|\nabla g|^2=\int_{\partial B_1}g g_r=\sum_{n=1}^{\infty}n(a_n^2+b_n^2).$$
In addition, $$\int_{\partial B_1}h_{\theta}^2=\sum_{n=1}^{\infty}n^2(a_n^2+b_n^2). $$
Hence, \begin{equation}\label{B1EC}\int_{B_1}|\nabla g|^2\leq\int_{\partial B_1}h_{\theta}^2.\end{equation}

The second part of the theorem follows from an integration by parts and the H\"older inequality.
\begin{eqnarray*}
\int_{\partial B_1}h^2&=&\int_{\partial B_1}(h^2 x)\cdot x=\int_{B_1}\nabla\cdot(g^2x)\\
&=&2\int_{B_1}g^2+2\int_{B_1}g\nabla g\cdot x\\
&\leq&2\int_{B_1}g^2+2(\int_{B_1}g^2)^{\frac{1}{2}}(\int_{B_1}|\nabla g|^2)^{\frac{1}{2}}\ \ \  (by\ |x|\leq1)\\
&\leq&C(\epsilon)\int_{B_1}g^2+\epsilon\int_{B_1}|\nabla g|^2\\
&\leq&C(\epsilon)\int_{B_1}g^2+\epsilon\int_{\partial B_1}h_{\theta}^2.\ \ \ \ \ \ \ \ \ \  \  \ \ \ \ \  \ \ \ \ \ \ \ \ (by\ (\ref{B1EC}))
\end{eqnarray*}
\end{proof}

Note that for the semiplanar graph $G$ with nonnegative curvature and any face $\sigma=\triangle_n$ of $G$, we have $3\leq n\leq D,$ $\frac{1}{\sqrt3}
\leq r_n=\frac{1}{\sin\frac{\pi}{n}}\leq \frac{1}{2\sin\frac{\pi}{D}}=C(D).$ Then the scaled version of Lemma \ref{B1ECL} reads
\begin{lemma} For $3\leq n\leq D$, and any Lipschitz function $h:\partial B_{r_n}\rightarrow \mathds{R},$ we denote by $g$ the harmonic function satisfying the Dirichlet boundary value problem
$$\left\{
\begin{array}{ll}
\Delta g=0,&in\ \mathring{B}_{r_n}\\
g|_{\partial B_{r_n}}=h&
\end{array}\right..$$ Then it holds that
\begin{equation}\label{ENGC1}\int_{\partial B_{r_n}}h^2\leq C(D,\epsilon)\int_{B_{r_n}}g^2+C(D)\epsilon\int_{\partial B_{r_n}}h_T^2,\end{equation} where $\epsilon$ is small, $T=\frac{1}{r_n}\partial_{\theta}$ is the unit tangent vector on the boundary $\partial B_{r_n}$ and $h_T$ is the directional derivative of $h$ in $T.$
\end{lemma}

The following lemma follows from the bi-Lipschitz property of the map $L_n:\triangle_n\rightarrow B_{r_n}.$
\begin{lemma}\label{EX1} Let $G$ be a semiplanar graph with $\Sec G\geq 0$ and $\sigma:=\triangle_n.$ Then we have
\begin{equation}\label{CONTROL}\sum_{y\in \partial \triangle_n\cap G}f^2(y)\leq C\int_{\partial \triangle_n}f_1^2\leq C(D)\int_{\triangle_n}\bar{f}^2.\end{equation}
\end{lemma}
\begin{proof} By the bi-Lipschitz property of $L_n$ and the inequality (\ref{ENGC1}), we have
\begin{equation}\label{ENGC2}\int_{\partial \triangle_n}f_1^2\leq C(D,\epsilon)\int_{\triangle_n}\bar{f}^2+C(D)\epsilon\int_{\partial \triangle_n}(f_1)_{T_n}^2,\end{equation} where $T_n$ is the unit tangent vector on the boundary $\partial \triangle_n.$
Let $e\subset\triangle_n$ be an edge with two incident vertices, $u$ and $v$. By linear interpolation, we have
$$\int_ef_1^2=\int_0^1(tf(u)+(1-t)f(v))^2dt=\frac{1}{3}(f(u)^2+f(u)f(v)+f(v)^2),$$ hence
\begin{equation}\label{ECTG1}\frac{1}{6}(f(u)^2+f(v)^2)\leq\int_ef_1^2\leq\frac{1}{2}(f(u)^2+f(v)^2).\end{equation} In addition,
\begin{equation}\label{ECTG2}\int_e(f_1)_{T_n}^2=(f(u)-f(v))^2\leq2(f(u)^2+f(v)^2).\end{equation}
Hence, by (\ref{ENGC2}) (\ref{ECTG1}) and (\ref{ECTG2}), we have
\begin{equation}\label{ECTG3}\int_{\partial \triangle_n}f_1^2\leq C(D,\epsilon)\int_{\triangle_n}\bar{f}^2+12C(D)\epsilon\int_{\partial \triangle_n}f_1^2.\end{equation} By setting $\epsilon=\frac{1}{24C(D)},$ (\ref{ECTG1}) and (\ref{ECTG3}) implies that
$$\sum_{y\in \partial \triangle_n\cap G}f^2(y)\leq C\int_{\partial \triangle_n}f_1^2\leq C(D)\int_{\triangle_n}\bar{f}^2.$$
\end{proof}

Let $G=(V,E,F)$ be a semiplanar graph with $\Sec G\geq0.$ For any $p\in X,$ there exists a face $\sigma\in F$ such that $p\in\sigma.$ For any vertex $q\in\sigma\cap G,$ we have $d(p,q)\leq C_3(D),$ since $\diam\sigma\leq C_3(D)$ for $\deg(\sigma)\leq D.$ Note that $3\leq d_x\leq6,$ for any $x\in G.$
\begin{lemma}\label{Lem1} Let $G$ be a semiplanar graph with $\Sec G\geq0.$ Then there exists a constant C(D) such that for any $p\in X,$ $q\in G$ on same face, we have
\begin{equation}\label{Lem11}|B_{r'}(p)|\leq C(D)|B_r^G(q)|,\end{equation}
where $r>\frac{2C_3(D)}{C},$ $r'=Cr-2C_3(D),$ and $C$ is the constant in (\ref{LEM2}).
\end{lemma}
\begin{proof} Let $r'=Cr-2C_3(D)>0,$ $p\in\sigma_0\in F$ and $q\in\sigma_0.$ We denote $W_{r'}:=\{\sigma\in F\mid \sigma\cap B_{r'}(p)\neq\emptyset\}$ and $\overline{W_{r'}}:=\bigcup_{\sigma\in W_{r'}}\sigma.$ It is obvious that $B_{r'}(p)\subset \overline{W_{r'}}.$ For any vertex $x\in \overline{W_{r'}}\cap G,$ there exists a face $\sigma_1\in W_{r'}$ such that $x\in \sigma_1,$ so that
\begin{eqnarray*}d(q,x)&\leq&d(p,x)+d(p,q)\\
&\leq&r'+\diam \sigma_1+\diam \sigma_0\leq r'+2C_3(D)\\
&=&C r
\end{eqnarray*}
Hence by (\ref{LEM2}) we have $d^G(q,x)\leq r$ which implies that \begin{equation}\label{VC33}\overline{W_{r'}}\cap G\subset B_r^G(q).\end{equation}
Since $3\leq\deg(\sigma)\leq D,$ $|\sigma|:=H^2(\sigma)\leq C(D).$ Then
\begin{equation}\label{VC11}|B_{r'}(p)|\leq |\overline{W_{r'}}|=\sum_{\sigma\in W_{r'}}|\sigma|\leq C(D)\sharp W_r,\end{equation} where $\sharp W_{r'}$ is the number of faces in $W_{r'}.$
Moreover,
\begin{equation}\label{VC22}3\sharp W_{r'}\leq \sum_{\sigma\in W_{r'}}\deg(\sigma)\leq \sum_{x\in \overline{W_{r'}}\cap G}d_x\leq 6\sharp (\overline{W_{r'}}\cap G),\end{equation} where $\sharp (\overline{W_{r'}}\cap G)$ is number of vertices in $\overline{W_{r'}}\cap G.$ Hence the lemma follows from (\ref{VC11}) (\ref{VC22}) and (\ref{VC33}),
$$|B_{r'}(p)|\leq C(D)\sharp W_{r'}\leq C(D)\sharp (\overline{W_{r'}}\cap G)\leq C(D)\sharp B_r^G(q)\leq C(D)|B_r^G(q)|.$$
\end{proof}

Now we can prove the mean value inequality for the extended function $\bar{f}$ defined on $X:=S(G)$ for some harmonic function $f$ on $G.$
\begin{proof}[Proof of Theorem \ref{MVIIL}]
For any $p\in X,$ there exists a face $\triangle_n$ such that $p\in\triangle_n.$ Then by the construction of $\bar f$ (see (\ref{DFF1})(\ref{DFF2})), there exists a vertex $q\in\partial \triangle_n\cap G$ such that
\begin{eqnarray}\label{meanv1}\bar{f}^2(p)&\leq&f^2(q)\nonumber\\
&\leq&\frac{C_2(D)}{|B_{C_1R}^G(q)|}\sum_{y\in B_{C_1R}^G(q)}f^2(y)d_y,
\end{eqnarray} where the last inequality follows from the mean value inequality (\ref{MVI}) for harmonic functions on the graph $G$.

By (\ref{Lem11}) in Lemma \ref{Lem1}, $$|B_{C_1R}^G(q)|\geq C(D)|B_{r'}(p)|,$$ where $r'=CC_1R-2C_3(D)\geq C(D)R,$ if $R\geq R_1(D).$ Hence
\begin{eqnarray}\label{meanv2}|B_{C_1R}^G(q)|&\geq& C|B_{CR}(p)|\nonumber\\
&\geq&C|B_{2C_1R}(p)|,
\end{eqnarray} the last inequality follows from the relative volume comparison (\ref{RVCA1}) on $X$.

Let $W_R:=\{\sigma\in F\mid\sigma\cap B_{C_1R}^G(q)\neq\emptyset\}$ and $\overline{W_R}:=\bigcup_{\sigma\in W_R}\sigma.$ For any $x\in \overline{W_R},$ there exist a face $\sigma_1\in W_R$ such that $x\in\sigma_1$ and a vertex $z\in B_{C_1R}^G(q)\cap \sigma_1.$ Then by (\ref{LEM2})
$$d(q,x)\leq d(q,z)+d(z,x)\leq d^G(q,z)+\diam \sigma_1\leq C_1R+C_3(D).$$
Hence $$\overline{W_R}\subset B_{C_1R+C_3(D)}(q)\subset B_{C_1R+2C_3(D)}(p)\subset B_{2C_1R}(p)$$ if $R\geq R_2(D).$
By (\ref{meanv1}) and (\ref{meanv2}), we obtain
\begin{eqnarray*}\bar{f}^2(p)&\leq&\frac{C_2}{|B_{2C_1R}(p)|}\sum_{y\in \overline{W_R}\cap G}f^2(y)\\
&\leq&\frac{C_2}{|B_{2C_1R}(p)|}\sum_{\sigma\in W_R}\sum_{y\in\partial \sigma\cap G}f^2(y)\\
&\leq&\frac{C_2}{|B_{2C_1R}(p)|}\sum_{\sigma\in W_R}\int_{\sigma}\bar{f}^2\\
&\leq&\frac{C_2}{|B_{2C_1R}(p)|}\int_{B_{2C_1R}(p)}\bar{f}^2,
\end{eqnarray*} if $R\geq R_2(D),$  where the last second inequality follows from (\ref{CONTROL}) in Lemma \ref{EX1}. Then the theorem follows by setting the new $R_1(D):=2C_1\max\{R_1(D), R_2(D)\}$.

\end{proof}

\section{Optimal Dimension Estimate}
In this section, we  estimate the dimension of the space of polynomial growth harmonic functions on a semiplanar graph with nonnegative combinatorial curvature.

Let $G$ be a semiplanar graph with $\Sec G\geq0.$ For some fixed $p\in G,$ we denote by $H^d(G):=\{u:G\rightarrow\mathds{R}\mid Lu=0, |u(x)|\leq C(d^G(p,x)+1)^d\}$ the space of polynomial growth harmonic functions on $G$ with growth rate less than or equal to $d.$ By the method of Colding-Minicozzi, the volume doubling property (\ref{GVD1}) and the Poincar\'e inequality (\ref{PII1})
imply that $\dim H^d(G)\leq C(D)d^{v(D)}$ for $d\geq1,$ where $C(D)$ and $v(D)$ are constants depending on the maximal facial degree $D$ of $G.$ Hua-Jost-Liu \cite{HJL}
used the weak relative volume comparison (\ref{GRV1}) on the graph $G$
and the Poincar\'e inequality to obtain the dimension estimate $\dim
H^d(G)\leq C d^2.$ But the optimal dimension estimate is linear in $d$
as in the Riemannian case (see \cite{CM3,CM4,L1}). On the graph $G,$
it is hard to obtain a nice relative volume comparison. But on the Alexandrov space $X:=S(G),$ the relative volume comparison (\ref{RVCA1}) follows from the Bishop-Gromov volume comparison theorem. To obtain the asymptotically optimal dimension estimate, we argue on the Alexandrov space $X$ instead of $G$.

We denote by $P^d(X):=\{u:X\rightarrow \mathds{R}\mid |u(x)|\leq C(d(p,x)+1)^d\}$ the space of polynomial growth functions on $X$ with growth rate less than or equal to $d.$  For any harmonic function on $G,$ we extend it to the function $\bar{f}$ defined on $X$ in the process of (\ref{DFF1}) and (\ref{DFF2}) which establishes a map
$$E:H^d(G)\rightarrow P^d(X),$$$$f\mapsto Ef=\bar{f}.$$ It is easy to see that $E$ is an injective linear operator. Hence it suffices to get the dimension estimate of the image $E(H^d(G)).$ By the relative volume comparison (\ref{RVCA1}) on $X$ and the mean value inequality (\ref{MVIOX}) for each function in $E(H^d(G)),$ we obtain the optimal dimension estimate (see \cite{L1,L2,CM4,Hu2}).

\begin{lemma}\label{innerprod} For any finite dimensional subspace $K\subset E(H^d(G)),$ there exists a constant $R_0(K)$ depending on $K$ such that for any $R\geq R_0(K),$
\begin{equation}\label{innerp1}A_R(u,v)=\int_{B_R(p)}uv\end{equation} is an inner product on $K.$
\end{lemma}

\begin{lemma}\label{PNGL1} Let $G$ be a semiplanar graph with $\Sec G\geq0,$ $K$ be a $k$-dimensional subspace of $E(H^d(G)).$ Given $\beta>1,\delta>0,$ for any $R_1\geq R_0(K)$ there exists $R>R_1$ such that if $\{u_i\}_{i=1}^k$ is an orthonormal basis of $K$ with respect to the inner product $A_{\beta R},$ then $$\sum_{i=1}^k A_R(u_i,u_i)\geq k\beta^{-(2d+2+\delta)}.$$
\end{lemma}

The following lemma follows from the mean value inequality (\ref{MVIOX}) for the extended functions.
\begin{lemma}\label{PNGL2} Let $G$ be a semiplanar graph with $\Sec G\geq0,$ $K$ be a $k$-dimensional subspace of $E(H^d(G)).$ Then there exists a constant
$C(D)$ such that for any fixed $0<\epsilon<\frac{1}{2}$, any basis of $K,$ $\{u_i\}_{i=1}^k,$ $R\geq R_2(D,\epsilon),$ where $\epsilon R_2\geq R_1(D)$ ($R_1(D)$ is the constant in Theorem \ref{MVIIL}) we have
$$\sum_{i=1}^kA_R(u_i,u_i)\leq C(D)\epsilon^{-1}\sup_{u\in <A,U>}\int_{B_{(1+\epsilon)R}(p)}u^2,$$ where $<A,U>:=\{w=\sum_{i=1}^ka_iu_i:\sum_{i=1}^ka_i^2=1\}.$
\end{lemma}
\begin{proof} For any $x\in B_R(p),$ we set $K_x=\{u\in K: u(x)=0\}.$ It
  is easy to see that $\dim K/K_x\leq1.$ Hence there exists an orthonormal linear transformation $\phi:K\rightarrow K,$ which maps $\{u_i\}_{i=1}^k$ to $\{v_i\}_{i=1}^k$ such that $v_i\in K_x,$ for $i\geq2.$ For any $x\in B_R(p),$ since $\epsilon R\geq \epsilon R_2\geq R_1,$ then $(1+\epsilon)R-r(x)\geq R_1$ for $r(x)=d(p,x).$ Hence the mean value inequality (\ref{MVIOX}) implies that for any $x\in B_R(p)$
\begin{eqnarray}\label{SOM1}
\sum_{i=1}^ku_i^2(x)&=&\sum_{i=1}^kv_i^2(x)=v_1^2(x)\nonumber\\
&\leq&C(D)|B_{(1+\epsilon)R-r(x)}(x)|^{-1}\int_{B_{(1+\epsilon)R-r(x)}(x)}v_1^2\nonumber\\
&\leq&C(D)|B_{(1+\epsilon)R-r(x)}(x)|^{-1}\sup_{u\in <A,U>}\int_{B_{(1+\epsilon)R}(p)}u^2.
\end{eqnarray}
For simplicity, denote $V_p(t)=|B_t(p)|$ and $A_p(t)=|\partial B_t(p)|.$

By the relative volume comparison (\ref{RVCA1}), we have
$$V_x((1+\epsilon)R-r(x))\geq\left(\frac{(1+\epsilon)R-r(x)}{2R}\right)^2V_x(2R)\geq\left(\frac{(1+\epsilon)R-r(x)}{2R}\right)^2V_p(R).$$
Hence, substituting it into (\ref{SOM1}) and integrating over
$B_R(p)$, we have
\begin{equation}\label{plg22}
\Sigma_{i=1}^k\int_{B_R(p)}u_i^2\leq \frac{C(D)}{V_p(R)}\sup_{u\in
\langle
A,U\rangle}\int_{B_{(1+\epsilon)R}(p)}u^2\int_{B_R(p)}(1+\epsilon-R^{-1}r(x))^{-2}dx
\end{equation}
Define $f(t)=(1+\epsilon-R^{-1}t)^{-2},$ then
$f'(t)=\frac{2}{R}(1+\epsilon-R^{-1}t)^{-3}\geq0,$
$$\int_{B_R(p)}f(r(x))dx=\int_0^Rf(t)A_p(t)dt.$$ Since $A_p(t)=V^{'}_p(t)\ a.e.,$
we integrate by parts and obtain
$$\int_0^Rf(t)A_p(t)dt=f(t)V_p(t)\mid_0^R-\int_0^RV_p(t)f'(t)dt.$$
Noting that $f^{'}(t)\geq0$ and the relative volume comparison
(\ref{RVCA1}), we have
\begin{eqnarray*}
\int_0^RV_p(t)f'(t)dt&\geq&\frac{V_p(R)}{R^2}\int_0^Rt^2f'(t)dt\\
&=&\frac{V_p(R)}{R^2}\{t^2f(t)\mid_0^R-2 \int_0^R t f(t)dt\}
\end{eqnarray*}
Therefore
\begin{equation*}
\int_{B_R(p)}f(r(x))dx\leq\frac{2V_p(R)}{R^2}\int_0^Rtf(t)dt
\leq2V_p(R)\epsilon^{-1}.
\end{equation*}
Combining this with \eqref{plg22}, we prove the
lemma.

\end{proof}

\begin{proof}[Proof of Theorem \ref{MTAL}] For any $k-$dimensional subspace $K\subset E(H^d(G)),$ we set $\beta=1+\epsilon,$ for fixed small $\epsilon$. By Lemma \ref{PNGL1}, there exists infinitely many $R>R_0(K)$ such that for any orthonormal basis $\{u_i\}_{i=1}^k$ of $K$ with respect to $A_{(1+\epsilon) R},$ we have
$$\sum_{i=1}^k A_R(u_i,u_i)\geq k(1+\epsilon)^{-(2d+2+\delta)}.$$
Lemma \ref{PNGL2} implies that $$\sum_{i=1}^k A_R(u_i,u_i)\leq C(D)\epsilon^{-1}.$$ Setting $\epsilon=\frac{1}{2d},$ and letting $\delta\rightarrow 0,$ we obtain
\begin{equation}\label{PF11}k\leq C(D)\left(\frac{1}{2d}\right)^{-1}\left(1+\frac{1}{2d}\right)^{2d+2+\delta}\leq C(D)d.\end{equation}

By (\ref{PF11}) and Theorem 1.4 in \cite{HJL} that $\dim H^d(G)=1$ for any $\Sec G\geq0,$ $D\geq 43$ and $d>0,$ we obtain
$$\dim H^d(G)\leq Cd.$$
\end{proof}

At the end, we use the Harnack inequality (\ref{HNKIG1}) in Lemma \ref{HNKIG} to prove Nayar's theorem \cite{N}. We denote by $H^d_{+}(G):=\{u:G\rightarrow\mathds{R}\mid Lu=0, u(x)\geq -C(d^G(p,x)+1)^d\}$ the set of one-side bounded polynomial growth harmonic functions with growth rate less than or equal to $d.$ This is not a linear space, but the linear span of $H^d_{+}(G),$ denoted by $\Span H^d_{+}(G),$ trivially contains $H^d(G).$ The following corollary implies that they are equal.
\begin{corollary} Let $G$ be a semiplanar graph with $\Sec G\geq0.$ Then $$\Span H^d_{+}(G)=H^d(G)$$ which implies that $$\dim\Span H^d_{+}(G)\leq Cd,$$ for $d\geq1.$
\end{corollary}

\begin{proof} It suffices to show $H^d_{+}(G)\subset H^d(G).$ For any $f\in H^d_{+}(G),$ there exists a constant $C$ such that
$f(x)\geq -C(d(p,x)+1)^d.$ We need to prove that $f(x)\leq C(d(p,x)+1)^d,$ for some $C.$ For simplicity, we assume $f(p)=0.$ Let $C_1(D)$ be the constant for the Harnack inequality in the Lemma \ref{HNKIG}. Then for any $x\in B_R^G(p),$ $R>0,$ it is easy to see that $B_{C_1R}^G(x)\subset B_{(C_1+1)R}^G(p).$ Moreover $$f(y)\geq -C(d(p,y)+1)^d\geq -C((C_1+1)R+1)^d\geq -CR^d,$$ for $y\in B_{(C_1+1)R}^G(p),$ $R\geq R_1(D).$ That is $f(y)+CR^d\geq0$ on $B_{C_1R}^G(x).$ The Harnack inequality (\ref{HNKIG1}) implies that
$$f(x)+CR^d\leq C(f(p)+CR^d)=CR^d.$$ Then we have
$$f(x)\leq CR^d,$$ for $x\in B_R^G(p),$ $R\geq R_1(D).$ Hence there exists a constant $C$ such that $f(x)\leq C(d(p,x)+1)^d.$
\end{proof}


\begin{thebibliography}{00}

\bibitem{BBI}\textit{D. Burago, Yu. Burago} and \textit{S. Ivanov}, A course in metric geometry, Graduate Studies in Mathematics \textbf{33},
American Mathematical Society, Providence, RI, 2001.

\bibitem{BGP} \textit{Yu. Burago, M. Gromov} and \textit{G. Perelman}, A. D. Aleksandrov spaces with curvatures bounded
below, Russ. Math. Surv. \textbf{47} (1992), 1-58.

\bibitem{BP1} \textit{O. Baues} and \textit{N. Peyerimhoff}, Curvature and geometry of tessellating plane graphs, Discrete Comput. Geom. \textbf{25} (2001), 141-159.

\bibitem{BP2} \textit{O. Baues} and \textit{N. Peyerimhoff}, Geodesics in non-positively curved plane tessellations,
Advances of Geometry \textbf{6} (2006), no. 2, 243-263.

\bibitem{Ch} \textit{B. Chen}, The Gauss-Bonnet formula of polytopal manifolds and the characterization of embedded graphs with nonnegative curvature, Proc. Amer. Math. Soc. \textbf{137} (2009), no. 5, 1601-1611.


\bibitem{CC} \textit{B. Chen} and \textit{G. Chen}, Gauss-Bonnet formula, finiteness condition, and characterizations of graphs embedded in surfaces, Graphs Combin. \textbf{24} (2008), no. 3, 159-183.

\bibitem{CW} \textit{R. Chen} and \textit{J. Wang}, Polynomial growth solutions to higher-order linear elliptic equations and systems, Pacific J. Math. \textbf{229} (2007), no. 1, 49-61.

\bibitem{CheY} \textit{S. Y. Cheng} and \textit{S. T. Yau}, Differential equations on Riemannian manifolds and their geometric applications,
Comm. Pure Appl. Math. \textbf{28} (1975), 333-354.

\bibitem{Chu} \textit{F. R. K. Chung}, Spectral graph theory, CBMS Regional Conference Series in Mathematics, \textbf{92}. Published for the Conference Board of the Mathematical Sciences, Washington, DC; by the American Mathematical Society, Providence, RI, 1997.

\bibitem{CM1} \textit{T. H. Colding} and \textit{W. P. Minicozzi II}, Harmonic functions with polynomial growth, J. Diff. Geom. \textbf{46} (1997), no. 1, 1-77.


\bibitem{CM2} \textit{T. H. Colding} and \textit{W. P. Minicozzi II}, Harmonic functions on manifolds, Ann. of Math. (2) \textbf{146} (1997), no.
3, 725--747.


\bibitem{CM3} \textit{T. H. Colding} and \textit{W. P. Minicozzi II}, Weyl type bounds for harmonic functions, Invent. Math. \textbf{131} (1998), no. 2, 257-298.


\bibitem{CM4} \textit{T. H. Colding} and \textit{W. P. Minicozzi II}, Liouville theorems for harmonic sections and applications, Comm. Pure Appl. Math. \textbf{51} (1998), no. 2, 113-138.

\bibitem{CoG} \textit{T. Coulhon} and \textit{A. Grigoryan}, Random walks on graphs with regular volume growth, Geom. Funct. Anal. \textbf{8} (1998), no. 4, 656-701.


\bibitem{D1} \textit{T. Delmotte}, Harnack inequalities on graphs, S\'eminaire de Th\'eorie Spectrale et G\'eom\'etrie, Vol. \textbf{16}, Ann\'ee 1997-1998, 217-228.

\bibitem{D2} \textit{T. Delmotte}, In\'egalit\'e de Harnack elliptique sur les graphes, (French) [Elliptic Harnack inequality on graphs] Colloq. Math. \textbf{72} (1997), no. 1, 19-37.


\bibitem{DM} \textit{M. DeVos} and \textit{B. Mohar}, An analogue of the Descartes-Euler formula for infinite graphs and Higuchi's conjecture, Trans. Amer. Math. Soc. \textbf{359} (2007), no. 7, 3287-3300 (electronic).


\bibitem{DK} \textit{J. Dodziuk} and \textit{L. Karp}, Spectral and function theory for
combinatorial Laplacians, Geometry of random motion (Ithaca, N.Y., 1987) 25-40, Contemp. Math., \textbf{73}, Amer. Math. Soc., Providence, RI, 1988.


\bibitem{DF} \textit{H. Donnelly} and \textit{C. Fefferman},
Nodal domains and growth of harmonic functions on noncompact manifolds,
J. Geom. Anal. \textbf{2} (1992), no. 1, 79-93.

\bibitem{G3} \textit{A. Grigor'yan}, Analysis on graphs, Lecture notes University of Bielefeld, 2009.

\bibitem{G}  \textit{M. Gromov}, Hyperbolic groups, Essays in group theory, 75-263, Math. Sci. Res. Inst. Publ., \textbf{8}, Springer, New York, 1987.


\bibitem{GS}  \textit{B. Gr\"unbaum} and \textit{G. C. Shephard}, Tilings and patterns, W. H. Freeman and Company, New York, 1987.


\bibitem{H} \textit{Y. Higuchi}, Combinatorial curvature for planar graphs,
 J. Graph Theory \textbf{38} (2001), no. 4, 220-229.


\bibitem{HS} \textit{I. Holopainen} and \textit{P. M. Soardi}, A strong Liouville theorem for $p$-harmonic functions on graphs, Ann. Acad. Sci. Fenn. Math. \textbf{22} (1997), no. 1, 205-226.


\bibitem{Hu1} \textit{B. Hua}, Generalized Liouville theorem in nonnegatively curved Alexandrov spaces, Chin. Ann. Math. Ser. B \textbf{30} (2009), no. 2, 111-128.


\bibitem{Hu2} \textit{B. Hua}, Harmonic functions of polynomial growth on singular spaces with nonnegative Ricci curvature, Proc. Amer. Math. Soc. \textbf{139} (2011), 2191-2205.

\bibitem{HJL} \textit{B. Hua, J. Jost} and \textit{S. Liu}, Geometric aspects of infinite semiplanar graphs with nonnegative curvature, preprint.


\bibitem{I}  \textit{M. Ishida}, Pseudo-curvature of a graph, lecture at "Workshop on topological graph theory", Yokohama National University, 1990.


\bibitem{KLe} \textit{S. W. Kim} and \textit{Y. H. Lee}, Polynomial growth harmonic functions on connected sums of complete Riemannian manifolds, Math. Z. \textbf{233} (2000), no. 1, 103-113.

\bibitem{K1} \textit{M. Keller}, The essential spectrum of the Laplacian on rapidly branching tessellations, Math. Ann. \textbf{346} (2010), no. 1, 51-66.

\bibitem{K2} \textit{M. Keller}, Cheeger constants, growth and spectrum of locally tessellating planar graphs, Math. Z. \textbf{268} (2011), no. 3-4, 871-886.

\bibitem{K3} \textit{M. Keller}, Curvature, geometry and spectral properties of planar graphs, Discrete \& Computational Geometry, \textbf{46} (2011) no. 3, 500-525.

\bibitem{Kl} \textit{B. Kleiner}, A new proof of Gromov's theorem on groups of polynomial growth, J. Amer. Math. Soc. \textbf{23} (2010), no. 3, 815-829.


\bibitem{KMS} \textit{K. Kuwae, Y. Machigashira} and \textit{T. Shioya}, Sobolev
spaces, Laplacian, and heat kernel on Alexandrov spaces, Math. Z.
\textbf{238} (2001), no. 2, 269-316.

\bibitem{Le} \textit{Y. H. Lee}, Polynomial growth harmonic functions on complete Riemannian manifolds, (English summary)
Rev. Mat. Iberoamericana \textbf{20} (2004), no. 2, 315-332.


\bibitem{L1}  \textit{P. Li}, Harmonic sections of polynomial growth, Math. Res. Lett. \textbf{4} (1997), no. 1, 35-44.


\bibitem{L2}  \textit{P. Li}, Harmonic functions and applications to complete manifolds (lecture notes), preprint.


\bibitem{LT} \textit{P. Li} and \textit{L.-F. Tam}, Complete surfaces with finite total curvature,
J. Differential Geom. \textbf{33} (1991), no. 1, 139¨C168.

\bibitem{LW3} \textit{P. Li} and \textit{J. Wang}, Counting massive sets and dimensions of harmonic functions, J. Differential Geom. \textbf{53} (1999), no. 2, 237-278.

\bibitem{LW1} \textit{P. Li} and \textit{J. Wang}, Mean value inequalities, Indiana Univ. Math. J. \textbf{48} (1999), no. 4, 1257-1283.

\bibitem{LW2} \textit{P. Li} and \textit{J. Wang}, Counting dimensions of L-harmonic functions, Ann. of Math. (2) \textbf{152} (2000), no. 2, 645-658.

\bibitem{N} \textit{P. Nayar}, On polynomially bounded harmonic functions on the $\mathds{Z}^d$ lattice, Bull. Pol. Acad. Sci. Math. \textbf{57} (2009), no. 3-4, 231-242.

\bibitem{RBK} \textit{T. R\'eti, E. Bitay} and \textit{Z. Kosztol\'anyi}, On the polyhedral graphs with positive combinatorial curvature, Acta Polytechnica Hungarica \textbf{2} (2005) no.2, 19-37.

\bibitem{St} \textit{D. Stone}, A combinatorial analogue of a theorem of Myers, Illinois J. Math. \textbf{20} (1976), 12-21.

\bibitem{SY} \textit{L. Sun} and \textit{X. Yu}, Positively curved cubic plane graphs are finite, J. Graph Theory \textbf{47} (2004), 241-274.


\bibitem{STW} \textit{C.-J. Sung, L.-F. Tam} and \textit{J. Wang} Spaces of harmonic functions, J. London Math. Soc. (2) \textbf{61} (2000), no. 3, 789-806.


\bibitem{T} \textit{L.-F. Tam}, A note on harmonic forms on complete manifolds, Proc. Amer. Math. Soc. \textbf{126} (1998), no. 10, 3097-3108.

\bibitem{W}
\textit{J. Wang}, Linear growth harmonic functions on complete manifolds, Comm. Anal. Geom.
\textbf{4} (1995), 683-698.

\bibitem{Woe}
\textit{W. Woess}, A note on tilings and strong isoperimetric inequality, Math. Proc.
Cambridge Philos. Soc. \textbf{124} (1998), 385-393.

\bibitem{Y1}
\textit{S. T. Yau}, Harmonic functions on complete Riemannian manifolds,
Comm. Pure Appl. Math. \textbf{28} (1975), 201-228.


\bibitem{Y2}
\textit{S. T. Yau}, Nonlinear analysis in geometry, Enseign. Math.
\textbf{33} (1987), (2), 109-158.


\bibitem{Y3}
\textit{S. T. Yau}, Differential geometry: Partial differential equations on
manifolds, Proc. of Symposia in Pure Mathematics, \textbf{54}, part 1, Ed. by
R.Greene and S.T. Yau, 1993.



\end{thebibliography}
\end{document}